\RequirePackage{fix-cm}
\documentclass[smallextended]{svjour3}      

\smartqed 
\usepackage[latin1]{inputenc}
\usepackage[english]{babel}
\usepackage[T1]{fontenc}
\usepackage{amssymb}
\usepackage{amsmath,amsfonts}
\usepackage{graphicx}
\usepackage{epstopdf}
\DeclareGraphicsExtensions{.pdf,.eps,.png,.jpg,.mps}
\usepackage{indentfirst}
\usepackage{dsfont}
\usepackage{color}
\usepackage{enumitem}
\usepackage{float}
\usepackage{pict2e}
\usepackage{soul}
\usepackage[OT2,OT1]{fontenc}

\newcommand{\id}{\mathsf{I}}

\newcommand{\Diag}{\operatorname{Diag}}
\newcommand{\tr}{\mathrm{Tr}}
\newcommand{\clo}{{\operatorname{cl}}}
\newcommand{\ope}{{\operatorname{op}}}

\newcommand\cyr
{
\renewcommand\rmdefault{wncyr}
\renewcommand\sfdefault{wncyss}
\renewcommand\encodingdefault{OT2}
\normalfont
\selectfont
}
\DeclareTextFontCommand{\textcyr}{\cyr}

\usepackage{hyperref}
\hypersetup{
    colorlinks,
    citecolor=black,
    filecolor=black,
    linkcolor=black,
    urlcolor=black
}

\begin{document}

\title{Enumerating simple paths from connected induced subgraphs}

\author{Pierre-Louis Giscard        \and Paul Rochet} 

\institute{P.L. Giscard \at
              University of York, Department of Computer Sciences \\
              \email{pierre-louis.giscard@york.ac.uk}           
           \and
           P. Rochet (corresponding author)\at
              Universit\'e de Nantes, Laboratoire de Math\'ematiques Jean Leray\\
							\email{paul.rochet@univ-nantes.fr}
}

\date{Received: date / Accepted: date}

\maketitle

\begin{abstract}
We present an exact formula for the ordinary generating series of the simple paths between any two vertices of a graph. Our formula involves the adjacency matrix of the connected induced subgraphs and remains valid on weighted and directed graphs. 
As a particular case, we obtain a relation linking the Hamiltonian paths and cycles of a graph to its dominating connected sets.
\end{abstract}

\section{Introduction}

Counting simple paths, that is trajectories on a graph that do not visit any vertex more than once, is a problem of fundamental importance in enumerative combinatorics \cite{madras2013self} with numerous applications, e.g. in sociology \cite{Ross1952,Cartwright1956}. Several general purpose methods for counting simple paths have been discovered over the last 60 years, which make use of the inclusion-exclusion principle \cite{bax1993inclusion,bax1994algorithms,Bjorklund2010,karp1982dynamic,Khomenko1972} or variants such as finite-difference sieves \cite{bax1996finite} and recursive expressions involving the adjacency matrix \cite{Alon1997,Harary1971,Perepechko2009,Ross1952}. More rare but also worth mentionning are approaches using different tools such as zeon algebras \cite{schott2011complexity} or immanantal equations \cite{cash2007number}.\\

While some of these theoretical results have been used to propose algorithms for counting simple cycles or paths of fixed length, the problem remains $\#$W[1]-complete and is generally beyond reach of existing techniques on real-world networks. A notable exception is the algorithm of Alon, Yuster and Zwick \cite{Alon1997}, designed to count simple cycles of length up to 7 on undirected graphs. Although extensions up to length 10 and directed graphs are possible in principle, the algorithm is memory intensive and cannot tackle arbitrary cycle lengths. 
Under additional conditions on the graph however, the problem can become manageable, e.g. it is fixed parameter tractable in bounded degree graphs.\\
 
Beyond the hunt for better asymptotic running times on worst case scenarios in problems closely related to simple cycle counting \cite{bjorklund2009counting,vassilevska2009finding}, much effort has also been devoted to the Hamiltonian path problem or its most famous variant, the Travelling Salesman problem. The classical dynamic programming solution of \cite{bellman1962dynamic} and \cite{held1962dynamic}, which runs at time $O(2^N)$ for $N$ the size of the graph, remains the faster known algorithm in the general case. However, some improvement is possible in graphs with bounded maximal degree, see for instance \cite{eppstein2007traveling,gebauer2008number} or \cite{bjorklund2008travelling}. In the latter, the problem is reduced to a search on the dominating connected subgraphs, a simplification that is central in the current article. 
\\

The contribution of this paper is two-fold. First, we present a matrix formulation of the inclusion-exclusion principle that provides a concise expression of the matrix generating series for simple paths and cycles in a labeled graph. This identity conciliates in a single equation the Hamiltonian path matrix expression of \cite{bax1993inclusion} or the number of simple cycles of arbitrary length \cite{Perepechko2009}. Second, we show a simplification of the formula that restricts the inclusion-exclusion process to connected induced subgraphs only. A remarkable consequence is an expression that links the Hamiltonian paths of a graph to its dominating connected sets. 
While the problem of counting simple cycles and paths parametrised by length remains $\#$W[1]-complete, the formulas we obtain form the base of a novel algorithm \cite{giscard2016general} that has proven to be efficient enough to effectively tackle the problem, up to length 20, on real-world networks \cite{giscard2016evaluating,giscard2017kernel}.\\

The paper is organized as follows. The matrix generating series $\mathsf P(z)$ of simple paths in a directed graph $G$ is introduced in Section \ref{sec:gen}, along with its expression in function of the labeled adjacency matrix of the induced subgraphs of $G$. We show in Section \ref{sec:connected} how $\mathsf P(z)$ can be expressd in function of connected induced subgraphs only, and the resulting relation linking the Hamiltonian paths of the graph to its connected dominating sets.

\section{The generating series of simple paths}\label{sec:gen}

Let $G = (V,E)$ be a directed graph with vertex set $V = \{ 1,...,N \}$ and edge set $E \subseteq V^2 $, which may contain self-loops. The directed edge, or \textit{arc}, from a vertex $i$ to a vertex $j$ is labeled $\omega_{ij}$. A path $p$ of length $\ell \geq 1$ is a sequence of $\ell$ contiguous arcs, that is, such that each new arc starts where the previous ended, e.g. $p= \omega_{i i_1} \omega_{i_1 i_2} ... \omega_{i_{\ell -1} j}$. Paths appear naturally through analytical transformations of the labeled adjacency matrix $\mathsf W$, with general term $\mathsf W_{ij} = \omega_{ij}$ if $(i,j) \in E$ and $\mathsf W_{ij} =0$ otherwise. Precisely, paths of a given length $k \geq 1$ are enumerated in the $k$-th power of $\mathsf W$:
\begin{equation} \label{wij} (\mathsf W^k)_{ij} = \sum_{\substack{p:\, i \to j \\ \ell(p) = k}} p \,,\quad i,j=1,...,N, \end{equation}
where the sum runs over all paths $p$ of length $\ell(p)=k$ from $i$ to $j$ on $G$. Replacing $\mathsf W$ by the (non-labeled) adjacency matrix $\mathsf A$, $ (\mathsf A^k)_{ij}$ simply counts the number of paths of length $k$ from $i$ to $j$.\\

A path $p= \omega_{i i_1} \omega_{i_1 i_2} ... \omega_{i_{\ell -1} j}$ is \textit{open} if its end vertices $i,j$ are different and \textit{closed} otherwise. A closed path is also called a \textit{cycle}. An edge $\omega_{ij}$ is a path of length one from $i$ to $j$ while self-loops $\omega_{ii}$ and backtracks $\omega_{ij} \omega_{ji}$ are cycles of length one and two respectively. By convention, the empty path $1$ is considered a cycle of zero length. Simple paths are paths that do not visit the same vertex more than once. Letting $\ell(p)$ denote the length of a path $p$, $V(p)$ the set of its vertices and $|V(p)|$ its size, simple paths can be characterized as the non-empty paths $p$ such that $\ell(p) = |V(p)|-1$ is $p$ is simple and $\ell(p) = |V(p)|$ if $p$ is closed. \\

In the literature, variants of the inclusion-exclusion principle led to discovering exact formulas for counting simple paths and cycles on graphs.  Exact formulas for small length paths  \cite{Ross1952,Harary1971,Alon1997} were later extended to paths of arbitrary length in \cite{Khomenko1972,Perepechko2009}. Surprisingly, these complicated general expressions somewhat simplify when focusing on Hamiltonian paths using a matrix form of the inclusion-exclusion principle, see e.g.~the formulas for $H_{n-1}$ and $H_n$ in \cite{bax1993inclusion}. In this spirit, a concise formula can be derived by considering the matrix $\mathsf P(z)$ whose $(i,j)$-entry is defined to be the ordinary generating function of simple paths from $i$ to $j$, i.e.
$$ 
\mathsf{P}_{ij}(z) = \sum_{\substack{p:\, i \to j \\ p \text{ simple} }} p\, z^{\ell(p)}, 
$$
for $z$ a formal variable. The information relative to simple paths and simple cycles on the digraph is entirely summarized in $\mathsf P(z)$, making it a natural object of interest. Let $\mathcal S = 2^V \setminus \emptyset$ be the set of non-empty subsets of $V$ (including $V$). 
For a matrix $\mathsf M$ indexed by the vertices of the graph (typically, the adjacency matrix $\mathsf A$ or the labeled version $\mathsf W$), define the restriction $\mathsf M_S$ of $\mathsf M$ to  $S \in \mathcal S$ by  
$$ 
\mathsf{M}_{S,ij} = \begin{cases}
\mathsf M_{ij} & \text{if $i,j \in S$}, \\
0 & \text{otherwise},
\end{cases}\quad i,j=1,...,N. 
$$

Let $\Diag(\mathsf M)$ denote the diagonal matrix obtained by setting to zero all non-diagonal entries in a square matrix $\mathsf M$. We deal separately with the open and close paths in $\mathsf P(z)$ by writing
$$
\mathsf P(z) = \mathsf P_{\clo}(z) + \mathsf P_{\ope}(z),
$$ 
where $\mathsf P_{\clo}(z) := \Diag \big(\mathsf P(z) \big)$ is the matrix generating series of simple cycles and $\mathsf P_{\ope}(z) := \mathsf P(z) -\mathsf P_{\clo}(z) $ is the generating matrix of open simple paths.
\begin{proposition}\label{th:matrix} It holds
\begin{itemize}
	\item[i)] $ \displaystyle \mathsf{P}_{\ope}(z) = \sum_{\substack{S \in \mathcal S}} (z\mathsf{W}_S)^{|S|-1}  (\id - z \mathsf{W}_S)^{N-|S|},$
	\item[ii)] $ \displaystyle \mathsf{P}_{\clo}(z) =\sum_{\substack{S \in \mathcal S}}  \Diag \Big( (z\mathsf{W}_S)^{|S|}  (\id - z \mathsf{W}_S)^{N-|S|} \Big)$.
\end{itemize}
\end{proposition}

\begin{proof} For any path $p$, it can be checked by direct calculation that
$$  \sum_{\substack{S \supseteq V(p)}} \binom{N- |S|}{\ell(p) +1 - |S|} (-1)^{\ell(p)+1 - |S|} = 
\begin{cases}1 & \text{if } \ell(p) = |V(p)|-1 \\ 
0 & \text{otherwise}
\end{cases} $$
with $\binom{n}{k}$ the binomial coefficient, set to zero for $k<0$ or $k>n$. Since $p$ is an open simple if, and only if $\ell(p) = V(p)-1$, the above expression provides an indicator function for open simple paths (in particular it is zero if $p$ is a cycle). Summing over all paths $p$ from $i$ to $j$ of fixed length $k >0$, then permuting the sums yields
$$ \sum_{\substack{p: i \to j \\ \ell(p) = k}} \!\! \mathds 1 \{ p \ \text{simple} \}  p= \sum_{\substack{S \in \mathcal S}}  \binom{N- |S|}{k+1 - |S|}   (-1)^{k+1 - |S|}\sum_{\substack{p: i \to j \\ \ell(p) = k \\ V(p) \subseteq S}} p. $$
The right-most sum recovers the $(i,j)$-entry of the $k$-th power of $\mathsf W_S$ by \eqref{wij}, which is trivially zero whenever $(i,j) \nsubseteq S$. The generating series of open simple paths follows by summing over all $k >0$ and once again permuting the summations:
$$ \mathsf P_{\ope} (z) = \sum_{\substack{S \in \mathcal S}}  \big(z\mathsf W_S \big)^{|S| -1}  \sum_{k =|S| }^{N-1} \binom{N- |S|}{k+1 - |S|} \big( -z \mathsf W_S \big)^{k+1 - |S|}. $$
We conclude by the binomial formula. The proof of $ii)$ is similar noting that for a non-empty cycle $c$,
$$ \sum_{\substack{S \supseteq V(c)}}\binom{N- |S|}{\ell(c) - |S|}  (-1)^{\ell(c) - |S|} = \mathds 1 \{ \ell(c) = |V(c)| \} = \mathds 1 \{ c \ \text{simple} \},  $$
and focusing on the diagonal terms. \qed 
\end{proof}

\begin{remark} Attributing the value $\omega_{ij}=1$ to all directed edges (thus replacing $\mathsf W$ by $\mathsf A$) in the expression of $\mathsf P_{\clo}(z)$ recovers with little work the formula in \cite{Perepechko2009,PVRussianArticle} on the number of simple cycles of length $k> 2$ on undirected graphs, namely
\begin{align*} \frac{1}{2k} \sum_{i=0}^k (-1)^{k-i}\binom{N-i}{N-k}\sum_{S:|S|=N-i}\tr \big( \mathsf A_S^k \big),
\end{align*}
where  $\tr(.)$ is the trace operator. Remark that, non-oriented cycles of length $k>2$ are counted twice (once in each direction) when viewing an undirected graph as a bi-directed digraph, which explains the factor $1/2$ in the above expression. One may even interpret the above sum as an enumeration of the simple cycles of length $k>2$ from every possible starting vertex, thus requiring a normalization of $1/2k$. Nevertheless, the simple proof of Perepechko and Voropaev's formula from Proposition \ref{th:matrix} remains valid on directed and weighted graphs.
\end{remark}

\section{Counting simple paths from weakly connected sets}\label{sec:connected}

A digraph is said to be weakly connected if replacing all its directed edges by undirected edges produces a connected undirected graph. The expression of $\mathsf P(z)$ can be reduced to a sum over weakly connected  induced subgraphs of $G$ owing to the simple property that the adjacency matrix of a disconnected digraph can be made block diagonal by an appropriate permutation of its indices. Let $G(S)$ denote the subgraph of $G$ induced by $S \in \mathcal S$. For all $S \in \mathcal S$, there is a unique partition $\mathcal C(S) = \{C_1,...,C_k \}$ dividing $G(S)$ into weakly connected components such that $G(S) = G(C_1) \cup ... \cup G(C_k)$. This partition verifies for all $n \geq 1$,
\begin{equation} \label{partition} \mathsf{W}_S^n = \mathsf{W}_{C_1}^n + \hdots + \mathsf{W}_{C_k}^n. \end{equation}
Let $\mathcal C = \mathcal C(V)\subseteq \mathcal S$ denote the non-empty subsets of $V$ for which the resulting induced subgraphs are weakly connected. For $C \in \mathcal C$, the weak neighborhood $N(C)$ of $C$ in $G$ is the set of vertices in $V \setminus C$ that can reach and/or be reached from $C$ in one step. Formally,
$$ N(C) = \{ i \in V \setminus C: \exists j \in C, (i,j) \in E \text{ and/or } (j,i) \in E \}.  $$
Of course, this definition recovers the classical definition of neighborhood in undirected graphs. 
\begin{theorem}\label{th:matrix_connected} The matrix generating series of open and closed simple paths verify:
\begin{itemize}
	\item[i)] $ \displaystyle \mathsf{P}_{\ope}(z) = \sum_{C \in \mathcal C} (z\mathsf{W}_C)^{|C|-1}  (\id - z \mathsf{W}_C)^{|N(C)|}$,
	\item[ii)] $ \displaystyle \mathsf{P}_{\clo}(z) = \sum_{C \in \mathcal C} \Diag \left( (z\mathsf{W}_C)^{|C|}  (\id - z \mathsf{W}_C)^{|N(C)|} \right)$.
\end{itemize}
\end{theorem} 

\begin{proof} Combining Theorem \ref{th:matrix} and Equation \eqref{partition} gives after permuting the sums
\begin{align*} \mathsf{P}_{\ope}(z)  
 =   \sum_{C \in \mathcal C} \sum_{S: C \in  \mathcal C(S)} (z\mathsf{W}_C)^{|S|-1}  (\id - z \mathsf{W}_C)^{N-|S|} \end{align*}
Fix $C \in \mathcal C$. A set $S \in \mathcal S$ such that $G(C)$ is a weakly connected component of $G(S)$ writes as $S = C \cup T$ for $T \subseteq V \setminus (C \cup N(C))$. Thus,
\begin{eqnarray*} 
& & \sum_{S: C \in  \mathcal C(S)} (z\mathsf{W}_C)^{|S|-1} (\id - z \mathsf{W}_C)^{N-|S|} \\
&  =& \!\!\! \sum_{T \subseteq V \setminus (C \cup N(C))} \hspace*{-0.5cm} (z\mathsf{W}_C)^{|C \cup T|-1}  (\id - z\mathsf{W}_{C})^{N - |C \cup T|} \\
& = & (z\mathsf{W}_C)^{|C|-1}  (\id - z\mathsf{W}_{C})^{|N(C)|} \hspace*{-0.5cm} \sum_{T \subseteq V \setminus (C \cup N(C))} \hspace*{-0.5cm} (z\mathsf{W}_C)^{|T|}  (\id - z\mathsf{W}_{C})^{N - |C| - |N(C)| - |T|}.
\end{eqnarray*}
Let $k = N - |C| - |N(C)|$, remark that
\begin{eqnarray*} \sum_{T \subseteq V \setminus (C \cup N(C))} \hspace*{-0.5cm} (z\mathsf{W}_C)^{|T|}  (\id - z\mathsf{W}_{C})^{k - |T|} = \sum_{j=0}^k \binom{k}{j} (z\mathsf{W}_C)^{j}  (\id - z\mathsf{W}_{C})^{k-j} = \id.
\end{eqnarray*}
Thus,
$$ \mathsf P_{\ope}(z) = \sum_{S \in \mathcal S} (z\mathsf{W}_S)^{|S|-1}  (\id - z \mathsf{W}_S)^{N-|S|} = \sum_{C \in \mathcal C}  (z\mathsf{W}_C)^{|C|-1}  (\id - z \mathsf{W}_C)^{|N(C)|}.  $$
The proof for $\mathsf P_{\clo}(z)$ is identical. \qed 
\end{proof}

From a computational point of view, the restriction to weakly connected induced subgraphs provides an improvement for counting simple paths if the graph $G$ contains relatively few connected induced subgraphs, e.g. if the graph is sparse. More precisely, it has been shown in \cite{giscard2016general}, that an algorithm based on the formulas of Theorem~\ref{th:matrix_connected} for counting simple cycles and paths of length up to $\ell$ achieves an asymptotic running time of 
$O\left(N+M+\big(\ell^\omega+\ell\Delta\big) |S_\ell|\right)$ and uses $O(N+M)$ space. In this expression,  $N$ is the number of vertices of the graph, $M$ is the number of edges,  $|S_\ell|$ is the number of (weakly) connected induced subgraphs of $G$ on at most $\ell$ vertices, $\Delta$ is the maximum degree of any vertex and $\omega$ is the exponent of matrix multiplication. Extensive comparisons with all existing techniques for counting simple cycles and paths \cite{giscard2016general}, show that the formulas of Theorem~\ref{th:matrix_connected} yield the best general purpose algorithm for this task whenever $(\ell^{\omega-1}\Delta^{-1}+1) |S_\ell|\leq |\text{Cycle}_\ell|$, with $|\text{Cycle}_\ell|$ the total number of simple cycles of length at most $\ell$, including backtracks and self-loops \cite{giscard2016general}. When this condition is not met the best general purpose algorithm is brute force search.

In conjunction with Monte Carlo sampling, the algorithm relying on Theorem~\ref{th:matrix_connected} has already permitted to count simple cycles of length up to 20 on 130,000+ vertices real-world networks \cite{giscard2016evaluating}. Furthermore, given that Theorem~\ref{th:matrix_connected} involves the labelled adjacency matrix $\mathsf{W}$, the formulas of the Theorem permit the \textit{enumeration} of the simple cycles and paths. By coding vertex labels using numerical values, 
this property was exploited to efficiently compare all label sequences corresponding to simple cycles in pairs of graphs, thereby reducing an important hurdle in automatic graph classification tasks \cite{giscard2017kernel}. \\

Let us now discuss the implications of this result on the Hamiltonian path problem. Remark that the terms of maximal degree in $\mathsf P(z)$ only involve  weakly connected sets $C$ for which $|C| + |N(C)| =N$, i.e. dominating sets. The reduction of the Hamiltonian path problem to dominating weakly connected sets has been investigated  in \cite{bjorklund2008travelling}, Theorem 3, where it proved to be a computational improvement for bounded degree graphs. In the sequel, let $\mathsf H$ be the Hamiltonian path counting matrix, whose $(i,j)$-entry gives the number of Hamiltonian paths from $i$ to $j$.

\begin{corollary}\label{cor:hamil} Let $\mathcal D$ be the set of weakly connected dominating sets in $G$,
$$ \mathsf H = \sum_{D \in \mathcal D} (-1)^{N-|D|} \Big( \mathsf{A}_D^{N-1} + \frac{\tr \big(\mathsf{A}_D^{N} \big)}{N} \ \id \Big). $$
\end{corollary}

 \begin{proof} Take $\mathsf W = \mathsf A$ in Theorem \ref{th:matrix_connected} and isolate the term of maximal degree. In $\mathsf P_{\ope}(z)$, this term writes
$$ \sum_{D \in \mathcal D} (-1)^{N-|D|} \mathsf{A}_D^{N-1}, $$
where we used that $|N(D)|=N-|D|$ due to the dominating property. This provides the off-diagonal part of $\mathsf H$. Its diagonal part follows similarly, noting that since simple cycles of length $N$ visit every vertex in the graph, all diagonal terms of maximal degree in $\mathsf P_\clo(z)$ are equal. \qed 
\end{proof}

A quick inspection of the proof reveals that, similarly as for $\mathsf P(z)$, open and closed Hamiltonian paths are dealt with separately yielding a slightly stronger version of the result, namely
$$ \mathsf H_{\ope} =  \sum_{D \in \mathcal D} (-1)^{N-|D|} \mathsf{A}_D^{N-1} \ \ \text{ and } \ \ \mathsf H_\clo = \frac 1 N \sum_{D \in \mathcal D} (-1)^{N-|D|} \tr \big(\mathsf{A}_D^{N} \big) \ \id,$$
recovering the matrices $H_{n-1}$ and $H_n$ in \cite{bax1993inclusion} with the summation restricted to connected dominating sets.

\section*{Acknowledgements}
P.-L. Giscard is grateful for the financial support from the Royal Commission for
the Exhibition of 1851. The authors are grateful to an anonymous referee for its many constructive remarks that helped improved the paper.


\end{document}